\documentclass[a4paper,12pt,fleqn]{article}
\usepackage[latin1]{inputenc}  \usepackage[T1]{fontenc}
\usepackage{lscape}
\usepackage{tabularx}     
\usepackage{pstricks}    
\usepackage{rotating}
\usepackage{amsfonts,amsmath,amsthm,amssymb,dsfont}
\usepackage{marginnote}
\usepackage{rotate}
\usepackage[arrow,matrix,curve]{xy} 	
\title{On the Classification of $LS$-Sequences}
\author{Christian Weiß}
\date{\today}

\newtheorem{thm}{Theorem}[section]
\newtheorem{defi}[thm]{Definition}
\newtheorem{rem}[thm]{Remark}
\newtheorem{lem}[thm]{Lemma}

\newtheorem{cor}[thm]{Corollary}

\newtheorem{exa}[thm]{Example}

\newcommand{\RR}{{\mathbb{R}}}
\newcommand{\ZZ}{{\mathbb{Z}}}
\newcommand{\NN}{{\mathbb{N}}}
\newcommand{\QQ}{{\mathbb{Q}}}

\makeindex

\begin{document} 

\maketitle

\begin{abstract}
This paper addresses the question whether the $LS$-sequences constructed in \cite{Car12} yield indeed a new family of low-discrepancy sequences. While it is well known that the case $S=0$ corresponds to van der Corput sequences, we prove here that the case $S=1$ can be traced back to symmetrized Kronecker sequences and moreover that for $S \geq 2$ none of these two types occurs anymore. In addition, our approach allows for an improved discrepancy bound for $S=1$ and $L$ arbitrary. 
\end{abstract}


\section{Introduction} \label{cha:Introduction} There are essentially three classical families of low-discrepancy sequences, namely Kronecker sequences, digital sequences and Halton sequences (compare \cite{Lar14}, see also \cite{Nie92}). In \cite{Car12}, Carbone constructed a class of one-dimensional low-discrepancy sequences, called $LS$-sequences with $L \in \NN$ and $S \in \NN_0$. The case $S=0$ corresponds to the classical one dimensional Halton sequences, called van der Corput sequences. However, the question whether $LS$-sequences indeed yield a new family of low-discrepancy sequences for $S \geq 1$ or if it is just a different way to write down already known low-discrepancy sequences has not been answered yet. In this paper, we address this question and thereby derive improved discrepancy bounds for the case $S=1$. 

\paragraph{Discrepancy.} Let $S=(z_n)_{n \geq 0}$ be a sequence in $[0,1)^d$. Then the \textbf{discrepancy} of the first $N$ points of the sequence is defined by
$$D_N(S) := \sup_{B \subset [0,1)^d} \left| \frac{A_N(B)}{N} - \lambda_d(B) \right|,$$
where the supremum is taken over all axis-parallel subintervals $B \subset [0,1)^d$ and $A_N(B) := \# \left\{ n \ \mid \ 0 \leq n < N, z_n \in B \right\}$ and $\lambda_d$ denotes the $d$-dimensional Lebesgue-measure. In the following we restrict to the case $d=1$. If $D_N(S)$ satisfies 
$$D_N(S) = O(N^{-1}\log N)$$
then $S$ is called a \textbf{low-discrepancy sequence}. In dimension one this is indeed the best possible rate as was proved by Schmidt in \cite{Sch72}, that there exists a constant $c$ with
	$$D_N(S) \geq c N^{-1} \log N.$$
The precise value of the constant $c$ is still unknown (see e.g. \cite{Lar14}). For a discussion of the situation in higher dimensions see e.g. \cite{Nie92}, Chapter~3.\\[12pt] 
A theorem of Weyl and Koksma's inequality imply that a sequence of points $(z_n)_{n \geq 0}$ is uniformly distributed if and only if
$$\lim_{N \to \infty} D_N(z_n) = 0.$$
Thus, the only candidates for low-discrepancy sequences are uniformly distributed sequences. A specific way to construct uniformly distributed sequences goes back to the work of Kakutani \cite{Kak76} and was later on generalized in \cite{Vol11} in the following sense. 
\begin{defi}
	Let $\rho$ denote a non-trivial partition of $[0,1)$. Then the \textbf{$\rho$-refinement} of a partition $\pi$ of $[0,1)$, denoted by $\rho \pi$, is defined by subdividing all intervals of maximal length positively homothetically to $\rho$. 
\end{defi}
Successive application of a $\rho$-refinement results in a sequence which is denoted by $\left\{ \rho^n \pi \right\}_{n \in \NN}$. The special case of \textbf{Kakutani's $\alpha$-refinement} is obtained by successive $\rho$-refinements where $\rho = \left\{ [0,\alpha), [\alpha, 1) \right\}$. If $\pi$ is the trivial partition $\pi = \left\{ [0,1) \right\}$ then we obtain \textbf{Kakutani's-$\alpha$-sequence}. In many articles Kakutani's $\alpha$-sequence serves as a standard example and the general results derived therein may be applied to this case (see e.g. \cite{CV07}, \cite{DI12}, \cite{IZ15}, \cite{Vol11}). Another specific class of examples of $\rho$-refinement was introduced in \cite{Car12}. 
\begin{defi} 
	Let $L \in \NN, S \in \NN_0$ and $\beta$ be the solution of $L \beta + S \beta^2 = 1$. An \textbf{$LS$-sequence of partitions} $\left\{ \rho_{L,S}^n \pi \right\}_{n \in \NN}$ is the successive $\rho$-refinement of the trivial partition $\pi = \left\{ [0,1) \right\}$ where $\rho_{L,S}$ consists of $L+S$ intervals such that the first $L$ intervals have length $\beta$ and the successive $S$ intervals have length $\beta^2$.
\end{defi}
The partition $\left\{ \rho_{L,S}^n \pi \right\}$ consists of intervals only of length $\beta^n$ and $\beta^{n+1}$. Its total number of intervals is denoted by $t_n$, the number of intervals of length $\beta^n$ by $l_n$ and the number of intervals of length $\beta^{n+1}$ by $s_n$. In \cite{Car12}, Carbone derived the recurrence relations 
\begin{align*}
t_n & = L t_{n-1} + S t_{n-2}\\
l_n & = L l_{n-1} + S l_{n-2}\\
s_n & = L s_{n-1} + S s_{n-2}
\end{align*}
for $n \geq 2$ with initial conditions $t_0 = 1, t_1= L+S, l_0 = 1, l_1 = L, s_0 = 0$ and $s_1 = S$. Based on these relations, Carbone defined a possible ordering of the endpoints of the partition yielding the $LS$-sequence of points. One of the observations of this paper is that this ordering indeed yields a simple and easy-to-implement algorithm but also has a certain degree of arbitrariness.
\begin{defi} \label{def:ordering_LS} Given an $LS$-sequence of partitions $\left\{ \rho_{L,S}^n \pi \right\}_{n \in \NN}$, the corresponding \textbf{$LS$-sequence of points} $(\xi^n)_{n \in \NN}$ is defined as follows: let $\Lambda_{L,S}^1$ be the first $t_1$ left endpoints of the partiton $\rho_{L,S} \pi$ ordered by magnitude. Given $\Lambda_{L,S}^n = \left\{ \xi_1^{(n)}, \ldots, \xi_{t_n}^{(n)} \right\}$ an ordering of $\Lambda_{L,S}^{n+1}$ is then inductively defined as
	\begin{align*}
	\Lambda_{L,S}^{n+1} = \left\{ \right.&  \xi_1^{(n)}, \ldots, \xi_{t_n}^{(n)},\\
	&\left. \psi_{1,0}^{(n+1)} (\xi_1^{(n)}), \ldots, \psi_{1,0}^{(n+1)} (\xi_{l_n}^{(n)}), \ldots, \psi_{L,0}^{(n+1)} (\xi_1^{(n)}), \ldots, \psi_{L,0}^{(n+1)} (\xi_{l_n}^{(n)}), \right. \\
	&\left. \psi_{L,1}^{(n+1)}	(\xi_1^{(n)}), \ldots, \psi_{L,1}^{(n+1)}	(\xi_{l_n}^{(n)}), \ldots, \psi_{L,S-1}^{(n+1)}	(\xi_1^{(n)}), \ldots, \psi_{L,S-1}^{(n+1)}	(\xi_{l_n}^{(n)}) \right\},
	\end{align*}
	where
	$$\psi^{(n)}_{i,j}(x) = x + i\beta^n + j\beta^{n+1}, \qquad x \in \RR.$$
\end{defi}
As the definition of $LS$-sequences might not be completely intuitive at first sight, we illustrate it by an explicit example.
\begin{exa} \label{exa:KF_sequence} For $L=S=1$ the $LS$-sequence coincides with the so-called \textbf{Kakutani-Fibonacci sequence} (see \cite{CIV14}). We have
\begin{align*}
	\Lambda^1_{1,1} & = \left\{ 0, \beta \right\}\\
	\Lambda^2_{1,1} & = \left\{ 0, \beta, \beta^2 \right\}\\
	\Lambda^3_{1,1} & = \left\{ 0, \beta, \beta^2, \beta^3, \beta + \beta^3 \right\}\\
	\Lambda^4_{1,1} & = \left\{ 0, \beta, \beta^2, \beta^3, \beta + \beta^3, \beta^4,\beta+\beta^4,\beta^2+\beta^4  \right\}
\end{align*}
	and so on.
\end{exa}

\begin{thm}[Carbone, \cite{Car12}] If $L \geq S$, then the corresponding $LS$-sequence has low-discrepancy.
\end{thm}
Carbone's proof is based on counting arguments but does not give explicit discrepancy bounds. These have been derived later by Iac\`{o} and Ziegler in \cite{IZ15} using so-called generalized $LS$-sequences. A more general result implicating also the low-discrepancy of $LS$-sequences can be found in \cite{AH13}. 
\begin{thm} \label{thm:IZ} [Iac\`{o}, Ziegler, \cite{IZ15}, Theorem 1, Section 3] If $(\xi_n)_{n \in \NN}$ is an $LS$-sequence with $L \geq S$ then 
	$$D_N(\xi_n) \leq \frac{B\log(N)}{N |\log(\beta)|} + \frac{B+2}{N},$$
	where 
	$$B = (2L+S-2) \left(\frac{R}{1-S\beta} +1 \right),$$
	with 
	$$R=\max \left\{ |\tau_1|,|\tau_1 + (L+S-2)\lambda_1|\right\},$$
	$\tau_1 = \frac{-L-2S+\sqrt{L^2+4S}}{2\sqrt{L^2+4S}}$ and $\lambda_1 = \frac{-L+\sqrt{L^2+4S}}{2\sqrt{L^2+4S}}$.	
\end{thm}
It has been pointed out that for parameters $S=0$ and $L=b$, the corresponding $LS$-sequence conincides with the classical van der Corput sequence, see e.g. \cite{AHZ14}.\footnote{If the reader is not familiar with the Definition of van der Coruput sequences, he may consult \cite{Nie92}, Section 3.1.} However, for higher values of $S$ it has been not been proved if $LS$-sequences indeed yield a new family of examples of low-discrepancy sequences or are just a new formulation of some of the well-known ones. We close this gap to a certain extent by showing the following main result:
\begin{thm} \label{main_thm} For $S=1$, the $LS$-sequences is a reordering of the symmetrized Kronecker sequences $(\left\{ n \beta \right\})_{n \in \ZZ}$. For $S \geq 2$ the $LS$-construction neither yields a (re-)ordering of a van der Corput sequence nor of a (symmetrized) Kronecker sequence.
\end{thm}
Let us make the notion of symmetrized Kronecker sequences more precise: given $z \in \RR$, let $\left\{ z \right\} := z - \lfloor z \rfloor$ denote the fractional part of $z$. A (classical) \textbf{Kronecker sequence} is a sequence of the form $(z_n)_{n \geq 0} = (\left\{ n z \right\})_{n \geq 0}$. If $z \notin \QQ$ and $z$ has bounded partial quotients in its continued fraction expansion (see Section \ref{cha:Proof}) then $(z_n)$ has low-discrepancy (\cite{Nie92}, Theorem 3.3). By a \textbf{symmtrized} Kronecker sequence we simply mean a sequence indexed over $\ZZ$ of the form $(\left\{ n z \right\})_{n \in \ZZ}$ with ordering
$$0, \left\{ z \right\}, \left\{ -z \right\}, \left\{ 2z \right\}, \left\{ -2z \right\}, \ldots$$ 
Note that it is still open, whether for $S\geq2$ an $LS$-sequence is a reordering of some other well-known low-discprancy sequence such as a digital-sequence or if the $LS$-construction really yields a new class of examples.\\[12pt]
Our approach does not only give a significantly shorter proof of low-discrepancy of $LS$-sequences for $L=1$ but also improves the known discrepancy bounds by Iac\'{o} and Ziegler in this case.
\begin{cor} \label{cor2} For $S=1$ the discrepancy of the $LS$-sequence $(\xi_n)_{n \in \NN}$ is bounded by
	$$D_N(\xi_n) \leq \frac{3}{N} + \left( \frac{1}{\log(\alpha)} + \frac{L}{\log(L+1)} \right) \frac{\log(N)}{N}$$
	where $\alpha = (1+\sqrt{5})/2$. 	
\end{cor}
Corollary~\ref{cor2} indeed improves the discrepancy bounds for $LS$-sequences given in Theorem~\ref{thm:IZ} in the specific case $S=1$. Both results yield inequalities of the type $$D_N(\xi_n) \leq \frac{\gamma}{N} + \frac{\delta \log(N)}{N}$$
For instance, if $L=S=1$ then Corollary~\ref{cor2} implies $\gamma = 3$ and $\delta = 2.776$ while according to Theorem~\ref{thm:IZ} the discrepancy can be bounded by $\gamma = 3.447$ and $\delta = 3.01$. The difference between the two results gets the more prominent the larger $L$ is: If $L=10$ and $S=1$ we get $\gamma = 3$ and $\delta = 5.51$ while Theorem~\ref{thm:IZ} only implies $\gamma = 22.87$ and $\delta = 9.03$.\footnote{We obtain different numerical values than in \cite{IZ15}. We checked our result on different computer algebra systems.} 

\section{Proof of the main results} \label{cha:Proof}

\paragraph{Continued fractions.} Recall that every irrational number $z$ has a uniquely determined infinite continued fraction expansion
$$z = a_0 + (1+a_1/(a_2+\ldots)) =: [a_0;a_1;a_2;\ldots],$$
where the $a_i$ are integers with $a_0 = \lfloor z \rfloor$ and $a_i \geq 1$ for all $i \geq 1$. The sequence of \textbf{convergents} $(r_i)_{i \in \NN}$ of $z$ is defined by
$$r_i = [a_0;a_1;\ldots;a_i].$$
The convergents $r_i = p_i/q_i$ with $\gcd(p_{i},q_{i}) = 1$ can also be calculated directly by the recurrence relation
\begin{align*}
& p_{-1} = 0, \qquad p_{0} = 1, \qquad p_{i} = a_ip_{i-1} + p_{i-2}, \quad i \geq 0\\
& q_{-1} = 1, \qquad q_{0} = 0, \qquad q_{i} = a_iq_{i-1} + q_{i-2}, \quad i \geq 0.
\end{align*}
\begin{rem} If $S=1$, then $\beta^2 + L\beta - 1 = 0$ or equivalently 
	$$\frac{1}{\beta} = L + \beta$$
	holds. Thus it follows that $a_i = L$ in the continued fraction expansion of $\beta$ for all $i=1,2,\ldots$.
\end{rem}
From now on the continued fraction expansion of $\beta$ is studied and it is always tacitly assumed, that the $q_i$'s are the denominators of the convergents of $\beta$.Although the proof of the following lemma is rather obvious we write it down here explictly because our proof of the main theorem is based on this arithmetic observation.

\begin{lem} \label{lem1} Let $n \in \NN_0$. If $S=1$ then we have
	\begin{itemize}
		\item[(i)] $\beta^{2n+1} + q_{2n} = q_{2n+1} \beta$.
		\item[(ii)] $\beta^{2n} - q_{2n-1} = - q_{2n} \beta$
	\end{itemize}
\end{lem}
\begin{proof}
	We prove both claims by induction.\\[12pt]
	(i) The identity is trivial for $n=0$. So we come to the induction step
	\begin{align*}
	\beta^{2n+1} + q_{2n} & = \beta^2\beta^{2n-1} + q_{2n} \left(\beta^2 + L\beta\right)\\
	& = \beta^2 \left(\beta^{2n-1} + q_{2n}\right) + Lq_{2n} \beta\\
	& = \beta^2 \left(q_{2n-1}\beta- q_{2n-2} + q_{2n} \right) + Lq_{2n} \beta\\
	& = \beta^2 \left(q_{2n-1}\beta + Lq_{2n-1} \right) + Lq_{2n}\beta\\
	& = q_{2n-1} \beta(\beta^2 + L\beta) + Lq_{2n}\beta\\
	& = q_{2n+1} \beta.
	\end{align*}
	(ii) The proof works analogously as in (i). We have $\beta^2 + 1 = -L\beta$ and
	\begin{align*}
	\beta^{2n} - q_{2n-1} & = \beta^2\beta^{2(n-1)} - q_{2n-1} \left(\beta^2 + L\beta\right)\\
	& = \beta^2 \left( \beta^{2(n-1)} - q_{2n-1} \right) - Lq_{2n-1} \beta\\
	& = \beta^2 \left( - q_{2n-2} \beta + q_{2n-3} - q_{2n-1} \right) - Lq_{2n-1} \beta\\
	& = \beta^2 \left( - q_{2n-2} \beta - Lq_{2n-2} \right) - Lq_{2n-1} \beta\\
	& = -q_{2n-2} \beta \left(\beta^2 + L\beta \right) - Lq_{2n-1} \beta\\
	& = -q_{2n} \beta.
	\end{align*}  
\end{proof}
\begin{exa} Consider the Kakutani-Fibonacci sequence from Example~\ref{exa:KF_sequence}. If we denote by $(f_n)_{n \geq 0}$ the Fibonacci sequence, i.e. the sequence inductively defined by $f_0 = 0, f_1 =1$ and $f_n = f_{n-1} + f_{n-2}$ for $n \geq 2$, we have that $q_i = f_i$ for all $i=1,2,\ldots$.	
\end{exa}
If $S=1$, then we can furthermore deduce from Definition~\ref{def:ordering_LS} that $t_{n+1} = t_n + L l_n$ and that $q_{n-1} = l_{n}$. Starting from $\xi_1$ we split the $LS$-sequence into consecutive blocks where the first block $B_1$ is of length $1$ and the $n$-th block $B_n$ for $n \geq 2$ is of length $L l_{n} = L q_{n-1} = t_{n} - t_{n-1}$. We now study the blocks $B_n$
\begin{align*}
B_n = & \psi_{1,0}^{(n)}	(\xi_1), \ldots, \psi_{1,0}^{(n)}	(\xi_{l_{n-1}}), \ldots, \psi_{L,0}^{(n)}(\xi_1), \ldots, \psi_{L,0}^{(n)}	(\xi_{l_{n-1}})\\
= & \ \xi_1 + \beta^{n-1}, \ldots, \xi_{l_{n-1}} + \beta^{n-1}, \ldots, \xi_1 + L \beta^{n-1}, \ldots, \xi_{l_{n-1}} + L \beta^{n-1}.
\end{align*}
\begin{lem} \label{lem2} Let $n \in \NN$. 
\begin{itemize}
	\item[(i)] If $n=2k+1$ is odd, then $B_n$ considered as a set consists of the $L \cdot q_{2k}$ elements $\left\{ -q_{2k-1} \beta\right\}, \left\{ -(q_{2k-1} + 1) \beta \right\}, \ldots, \left\{ -(q_{2k+1}-1)  \beta \right\}$ (respectively of the element $0$ if $n=1$).
	\item[(ii)] If $n=2k$ is even, then $B_n$ considered as a set consists of the $L \cdot q_{2k-1}$ elements $\left\{ (q_{2k-2}+1) \beta\right\}, \left\{ (q_{2k-2} + 2) \beta \right\}, \ldots, \left\{ q_{2k} \beta \right\}$.
\end{itemize}
\end{lem}
Before going into the rather technical details of the proof, let us explain its idea for the example of the Kakutani-Fibonacci sequence ($L=S=1$). This sequence of points is given by
$$\underbrace{0}_{B_0}, \underbrace{\beta}_{B_1}, \underbrace{\beta^2}_{B_2}, \underbrace{\beta^3,\beta+\beta^3}_{B_3}, \underbrace{\beta^4,\beta+\beta^4,\beta^2+\beta^4}_{B_4}, \ldots.$$
Using $\beta + \beta^2$ this can be easily re-written as
$$\underbrace{0}_{B_0}, \underbrace{\beta}_{B_1},  \underbrace{1-\beta}_{B_2}, \underbrace{2\beta - 1 ,3 \beta - 1}_{B_3}, \underbrace{2-3\beta, 2-2\beta, 3- 4\beta}_{B_4}, \ldots.$$
\begin{proof}
  The two assertions are proved simultaneously by induction on $k$. For $n=1,2$ the claim is obvious from definition, since $\xi_1 = 0$ and $\xi_2 = \beta, \ldots, \xi_L = L\beta$. Let $k \geq 2$ and $n=2k+1$ be odd. If we denote by $\equiv$ equivalence modulo $1$ we have for $m \in \left\{ 0, \ldots, l_{n-1} \right\}$ by Lemma~\ref{lem1} and induction hypothesis
  \begin{align*}
	  \xi_m + j\beta^{2k+1-1} \equiv \xi_m - jq_{2k}\beta \equiv (r - jq_{2k}) \beta,
  \end{align*}
  with $-q_{2k-1} +1 \leq r \leq -q_{2k-3}$ and $q_{2k-2} + 1\leq r \leq q_{2k}$ and $1 \leq j \leq L$. Thus it follows that
  \begin{align*}
	  &-q_{2k-1} + 1  -Lq_{2k}  \leq r - jq_{2k} \leq q_{2k} - q_{2k}\\
	  \Leftrightarrow & -(q_{2k+1} - 1)\leq r - jq_{2k} \leq 0.
  \end{align*}
  Since the sequence is injective, the claim follows for odd $n$. So let $n=2k+2$ be even. Then we use again Lemma~\ref{lem1} and induction hypothesis to derive
    \begin{align*}
    \xi_m + j\beta^{2k+2-1} \equiv \xi_m + jq_{2k+1}\beta \equiv (r + jq_{2k+1}) \beta,
    \end{align*}
  with $-q_{2k-1} +1 \leq r \leq -q_{2k-3}$ and $q_{2k-2} + 1 \leq r \leq q_{2k}$ and $1 \leq j \leq L$. This completes the induction since
    \begin{align*}
    &-q_{2k-1} + 1  + q_{2k+1}  \leq r + jq_{2k+1} \leq q_{2k} + Lq_{2k+1}\\
    \Leftrightarrow & 1 \leq r + jq_{2k+1} \leq q_{2k+2}.
    \end{align*}
\end{proof}
\begin{proof}[Proof of Theorem~\ref{main_thm}] If $S = 1$ the $LS$-sequence is indeed a reordering of the symmetrized Kronecker sequence by Lemma~\ref{lem2}. So let $S \geq 2$ and $L \geq S$. Then $\beta$ is irrational and the recurrence relation
	\begin{align} \label{eq1}
		\beta^{2} = \frac{1 - L\beta}{S}.
	\end{align}	
holds. Hence the $LS$-sequence cannot be a reordering of a van der Corput sequence (which consists only of rational number).\\[12pt] Now assume that the $LS$-sequence is the reordering of a (possibly symmetrized) Kronecker sequence $\left\{ n \alpha \right\}$ for some $\alpha \in \RR$. Since $\alpha$ itself has to be an element of the $LS$-sequence, there exists an $n \in \NN$ such that $\alpha$ can be uniquely written in the form
	$$\alpha = \sum_{k=1}^n \alpha_k \beta^k$$
with $\alpha_k \in \left\{0,\ldots,L\right\}$ for $k=1,\ldots,n$ and $\alpha_n \neq 0$. By~\eqref{eq1} we have the equality $\beta^k = x_k \beta + y_k$ with $x_k,y_k \in \QQ$ and $s^k x_k, s^k y_k \in \ZZ$. Thus, $\alpha$ itself can be rewritten as $\alpha = x_{\alpha} \beta  + y_{\alpha}$ with $x_{\alpha},y_{\alpha} \in \QQ$ and $s^n x_{\alpha}, s^n y_{\alpha} \in \ZZ$. However, $\beta^{n+1}$, which is an element of the $LS$-sequence, cannot be an element of $\left\{ n \alpha \right\}_n$ since $\beta^{n+1} = x_{n+1} \beta + y_{n+1}$, where at least one of $x_{n+1}$ and $y_{n+1}$ has denominator $s^{n+1}$. This is a contradiction.
\end{proof}
A main advantage of the approach via symmetrized Kronecker sequence is that it yields a possibility to calculate improved discrepancy bounds, namely Corollary~\ref{cor2}.
\begin{proof}[Proof of Corollary~\ref{cor2}]
	We imitate the proofs in \cite{Nie92}, Theorem 3.3 and \cite{KN74}, Theorem 3.4 respectively and leave away here the technical details that are explained therein very nicely: The number $N$ can be represented in the form
	$$ N = \sum_{i=0}^{l(N)} c_i q_i,$$
	where $l(N)$ is the unique non-negative integer with $q_{l(N)} \leq N < q_{l(N)+1}$ and where the $c_i$ are integers with $0 \leq c_i \leq L$. Let $LS_N$ denote the set consisting of the first $N$ numbers of the $LS$-sequence. We decompose $LS_N$ into blocks of consecutive terms, namely $c_i$ blocks of length $q_i$ for all $0 \leq i \leq l(N)$. Consider a block of length $q_i$ and denote the corresponding point set by $A_i$. If $i$ is odd, $A_i$ consists of the fractional parts $\left\{ nz \right\}$ with $n=n_i,n_i+1,\ldots,n_i+q_{i}-1$ according to Lemma~\ref{lem2}. As shown in the proof of \cite{Nie92}, Theorem 3.3., this point set has discrepancy
	$$D_{q_i}(A_i) < \frac{1}{q_{i-1}} + \frac{1}{q_{i}}.$$
	If $i$ is even, $A_i$ consists of the fractional parts $\left\{ -nz \right\}$ with again $n=n_i,n_i+1,\ldots,n_i+q_{i}-1$ by Lemma~\ref{lem2}. Since $z$ and $-z$ have the same continued fraction expansion up to signs, we also have
	$$D_{q_i}(A_i) < \frac{1}{q_{i-1}} + \frac{1}{q_{i}}.$$
	Analogous calculations as in \cite{KN74} then yield the assertion.
\end{proof}
Asymptotically we deduce the following behaviour, again improving the more general result of \cite{IZ15} in the special case $S=1$.
\begin{cor}
	If $S=1$, then we obtain
	$$\lim_{N \to \infty} \frac{ND_N(\xi_n)}{\log N} \sim \frac{L}{\log(L)}$$
	as $L \to \infty$.
\end{cor}
Finally, we would like to point out the fact that it follows immediately from our approach that the Kakutani-Fibonacci sequence is the reordering of an orbit of an ergodic interval exchange transformation. In \cite{CIV14}, it was shown that a much more complicated interval exchange transformation is necessary in order to get the original ordering given in Definition~\ref{def:ordering_LS}.
\begin{cor} For $L=1$, the $LS$-sequence is always a reordering of an orbit of an ergodic interval exchange transformation. \end{cor}
\begin{proof} The map $R_\alpha: x \mapsto x + \alpha \ (\mod 1)$, the rotation of the circle by $\alpha$, is ergodic for $\alpha \notin \QQ$, see e.g. \cite{EW11}, Example 2.2. Moreover, it is an interval exchange transformation, compare e.g. \cite{Via06}.
\end{proof}

\paragraph{Acknowledgments.} The author thanks Soumya Bhattacharya, Anne-Sophie Krah, Zoran Nikoli\'{c} and Florian Pausinger for their comments on an earlier version of this paper. Furthermore, I am grateful to the referee for his suggestions.

\textsc{Hochschule Ruhr West, Duisburger Str. 100, D-45479 M\"ulheim an der Ruhr}\\
\textit{E-mail address:} \texttt{christian.weiss@hs-ruhrwest.de}

\end{document}